\documentclass[reqno]{amsart}
\usepackage{allneeded}

\DeclareMathOperator*{\toup}{\longrightarrow} 
\newcommand{\K}{\mathds{K}}
\newcommand{\N}{\mathds{N}}
\newcommand{\M}{\mathds{M}}
\newcommand{\Z}{\mathds{Z}}
\newcommand{\R}{\mathds{R}}
\newcommand{\T}{\mathds{T}}
\newcommand{\crit}{\mathrm{crit}}
\newcommand{\kk}{\bm{k}}
\newcommand{\ev}{\mathrm{ev}}
\newcommand{\dist}{\mathrm{dist}}

\newcommand{\injrad}{\mathrm{injrad}}
\newcommand{\id}{\mathrm{id}}
\newcommand{\ind}{\mathrm{ind}}
\newcommand{\avind}{\overline{\ind}}
\newcommand{\diff}{\mathrm{d}}
\newcommand{\fix}{\mathrm{fix}}
\newcommand{\Hom}{\mathrm{H}}
\newcommand{\Tan}{\mathrm{T}}
\newcommand{\orb}{\mathrm{orb}}
\newcommand{\W}{W^{1,2}}
\newcommand{\Wloc}{\W_{\mathrm{loc}}}
\begin{document}

\title[Isometry-invariant geodesics and the fundamental group]{Isometry-invariant geodesics\\ and the fundamental group}

\author{Marco Mazzucchelli}
\address{UMPA, \'Ecole Normale Sup\'erieure de Lyon, CNRS, 69364 Lyon, France}%
\email{marco.mazzucchelli@ens-lyon.fr}%
\thanks{This research is partially supported by the  ANR projects WKBHJ (ANR-12-BS01-0020) and  COSPIN (ANR-13-JS01-0008-01).}


\date{October 23, 2013. \emph{Revised: }July 4, 2014}

\begin{abstract}
We prove that on closed Riemannian manifolds with infinite abelian, but not cyclic, fundamental group, any isometry that is homotopic to the identity possesses infinitely many invariant geodesics. We conjecture that the result remains true if the fundamental group is infinite cyclic. We also formulate a generalization of the isometry-invariant geodesics problem, and a generalization of the celebrated Weinstein conjecture: on a closed contact manifold with a selected contact form, any strict contactomorphism that is contact-isotopic to the identity possesses an invariant Reeb orbit.
\end{abstract}

\maketitle

\section{Introduction}

\subsection{Background and main result}
Since the work of Hadamard \cite{Hadamard:Les_surfaces_a_courbures_opposees_et_leurs_lignes_geodesiques} and Poin\-car\'e \cite{Poincare:Sur_les_lignes_geodesiques_des_surfaces_convexes}, the problem of the existence of closed geodesics has occupied a central place in Riemannian geometry. In the presence of symmetry, described by an isometry $I$ of the Riemannian manifold $(M,g)$, the analogous problem consists in searching for geodesics on which the isometry acts as a non-trivial translation: namely, geodesics $\gamma:\R\looparrowright M$ that, after being suitably reparametrized with constant positive speed, satisfy $I(\gamma(t))=\gamma(t+1)$ for all $t\in\R$. These geodesics are simply called \textbf{isometry-invariant} or, more specifically, \textbf{$I$-invariant}. Their study was initiated by Grove \cite{Grove:Condition_C_for_the_energy_integral_on_certain_path_spaces_and_applications_to_the_theory_of_geodesics, Grove:Isometry_invariant_geodesics} in the 1970s, and since then several existence and multiplicity results have been established (see \cite{Mazzucchelli:On_the_multiplicity_of_isometry_invariant_geodesics_on_product_manifolds} for some history of the problem).

In a beautiful short paper of 1984 \cite{Bangert_Hingston:Closed_geodesics_on_manifolds_with_infinite_Abelian_fundamental_group}, Bangert and Hingston proved that on any closed Riemannian manifold with infinite abelian fundamental group, the growth rate of closed geodesics is at least the one of prime numbers (in particular, there are infinitely many of them). This lower bound is obtained by a clever use of a simple non-divisibility argument: on the space of closed curves whose homotopy class is a prime multiple of a generator of infinite order of the fundamental group, a closed geodesic obtained by means of a suitable minimax scheme is not iterated. In the current paper, we investigate the generalization of Bangert and Hingston's  result to the case of geodesics invariant by an isometry homotopic to the identity. The situation, here, appears way more complex: the non-divisibility argument mentioned above does not seem to hold in the isometry-invariant setting. However,  a subtle study of the local properties of isometry-invariant geodesics still allows to recover a multiplicity result in case the fundamental group is not cyclic.

\begin{thm}\label{t:main}
Let $(M,g)$ be a closed connected Riemannian manifold whose fundamental group is isomorphic to $\Z\oplus H$, for some non-trivial abelian group $H$. Every isometry of $(M,g)$ homotopic to the identity admits infinitely many invariant geodesics.
\end{thm}

It seems reasonable to expect that the result still hold in case the fundamental group is isomorphic to $\Z$. The seemingly technical difficulty we come across when we relax the non-cyclicity assumption is that the available minimax scheme that produces isometry-invariant geodesics is not one-dimensional anymore: as in the closed geodesics case, we would need to perform minimax on a suitable family of non-trivial $n$-spheres, for some $n\geq 2$, inside the space of isometry-invariant paths. We refer the reader to Remark~\ref{r:infinite_cyclic_fundamental_group} for a more precise discussion of this issue.

\subsection{Invariant Reeb orbits}
The geodesic flow of $(M,g)$ can be seen as the Reeb flow on the unit cotangent bundle $\mathrm{S}^*M$ equipped with the contact form given by the restriction of the Liouville form. An isometry of $(M,g)$ lifts to a contactomorphism of $\mathrm{S}^*M$ that preserves the contact form. In this way, the problem of isometry-invariant geodesics can be seen as a special instance of the problem of invariant Reeb orbits, which goes as follows. Let $(Y,\xi)$ be a closed contact manifold, equipped with a contact form $\alpha$ such that $\ker\alpha=\xi$ and a contactomorphism $\phi$ that preserves $\alpha$.  We denote by $R$ the corresponding Reeb vector field, which is defined by $\alpha(R)\equiv1$ and $\diff\alpha(R,\cdot)\equiv0$. We call \textbf{$\phi$-invariant Reeb orbit} a curve $\gamma:\R\to Y$ such that $\dot\gamma= R\circ\gamma$ and, for some time-shift $T\neq0$,  satisfies $\phi(\gamma(t))=\gamma(t+T)$.  We wish to stress the importance of the requirement that the time-shift be non-zero: with this definition, $\id$-invariant Reeb orbits are precisely closed Reeb orbits. Therefore, the existence problem for invariant Reeb orbits contains the corresponding problem for closed Reeb orbits. It seems natural to formulate the following generalization of the celebrated Weinstein conjecture (see e.g.~\cite[page~120]{Hofer_Zehnder:Symplectic_invariants_and_Hamiltonian_dynamics}).

\begin{conj}\label{conjecture}
In a closed contact manifold with a selected contact form, every contactomorphism $\phi$ that preserves the form and is contact-isotopic to the identity admits a $\phi$-invariant Reeb orbit. 
\hfill\qed
\end{conj}

The notion of invariant Reeb orbit is related to the one of translated point, introduced by Sandon in \cite{Sandon:On_iterated_translated_points_for_contactomorphisms_of_R2n+1_and_R2nxS1}. Given a contactomorphism $\psi$ of the above contact manifold $(Y,\xi)$ that does not necessarily preserve the contact form $\alpha$, a translated point of $\psi$ is a point $x\in Y$ that belongs to the same Reeb orbit of its image $\psi(x)$ and satisfies $(\psi^*\alpha)_x=\alpha_x$. In the same paper, Sandon conjectured that any such $\psi$ that is contact-isotopic to the identity possesses at least as many translated points as the minimal number of critical points of a smooth real-valued function on $Y$. When $\psi$ preserves the contact form $\alpha$, its $\psi$-invariant Reeb orbits are precisely the Reeb orbits of translated points that are closed or not contained in the set of fixed points of $\psi$. If we further assume that $\psi$ does not have fixed points (which is true for a generic perturbation of $\psi$ by the Reeb flow) our Conjecture~\ref{conjecture}  becomes a special case of Sandon's one. However, we wish to stress that the two conjectures have quite a different flavor: indeed, ours contains the Weinstein conjecture, whereas this is not the case for Sandon's one.

Conjecture~\ref{conjecture} is known to be true in the geodesics setting:  in a closed Riemannian manifold, every isometry that is homotopic to the identity possesses an invariant geodesic, as Grove proved in \cite[Theorem~3.7]{Grove:Condition_C_for_the_energy_integral_on_certain_path_spaces_and_applications_to_the_theory_of_geodesics}. If the isometry is not homotopic to the identity but is fixed point free, an invariant geodesic can be found by a simple minimization procedure \cite[Proposition~2.7]{Grove:Condition_C_for_the_energy_integral_on_certain_path_spaces_and_applications_to_the_theory_of_geodesics}. However, the topology of the Riemannian manifold may force the existence of fixed points (e.g.\ if the Euler characteristic is non-zero), and the hypothesis that the isometry be homotopic to the identity cannot be completely relaxed, as the example of the isometry $I(x,y)=(1-y,x)$ on the flat torus $\T^2=[0,1]^2/\{0,1\}^2$ shows. 

More generally, Conjecture~\ref{conjecture} holds when the contact manifold is a unit cotangent bundle (equipped with the Liouville contact form) of a closed manifold whose fundamental group is either finite or has infinitely many conjugacy classes\footnote{There is no known example of a finitely generated infinite group with finitely many conjugacy classes.}. This statement is not explicitly proved in the literature, but can be inferred by combining the arguments in \cite{Albers_Merry:Translated_points_and_Rabinowitz_Floer_homology} and \cite{Kang:Survival_of_infinitely_many_critical_points_for_the_Rabinowitz_action_functional}. The author is grateful to Peter Albers and Will Merry for pointing this out to him.

\subsection{Organization of the paper}
In Section~\ref{s:Preliminaries} we discuss the features of the variational principle for isometry-invariant geodesics: in Sections~\ref{s:The_energy_function}, \ref{s:Topology_of_Lambda_M_I}, and~\ref{s:local_properties} we give the background on the Morse theory of the energy function, while in Section~\ref{s:Bangert} we provide a crucial technical result (Lemma~\ref{l:Bangert_for_I_invariant_geodesics}) on minimax isometry-invariant geodesics. In Section~\ref{s:multiplicity_results} we prove multiplicity results: in Section~\ref{s:multiplicity_results_1} we establish Theorem~\ref{t:main} in the easier case in which the rank of the fundamental group is at least two, and finally in Section~\ref{s:multiplicity_results_2} we carry out the proof of  Theorem~\ref{t:main} in the general case.

\section{Preliminaries}
\label{s:Preliminaries}

\subsection{The energy function}
\label{s:The_energy_function}

Let us recall the variational setting for the study of isometry-invariant geodesics from \cite{Grove:Condition_C_for_the_energy_integral_on_certain_path_spaces_and_applications_to_the_theory_of_geodesics}. We consider a closed connected Riemannian manifold $(M,g)$ of dimension larger than one, equipped with an isometry $I$. We denote by $\Lambda(M;I)$ the space of $\Wloc$ curves $\gamma:\R\to M$ such that $I(\gamma(t))=\gamma(t+1)$ for all $t\in\R$. This space, equipped with a suitable infinite-dimensional Riemannian metric, is a complete Hilbert manifold. The real line $\R$ acts on $\Lambda(M;I)$ by time-shift: given an $I$-invariant curve $\gamma$, a real number $\tau$ acts on it by $(\tau\cdot\gamma)(t)=\gamma(t+\tau)$. This action is not free on the subspace of periodic curves, the isotropy groups being generated by the minimal period of the curve if it is non-zero, or the whole $\R$ if the curve is stationary. 

We will be interested in the energy function $E:\Lambda(M;I)\to[0,\infty)$, which is given by
\[
E(\gamma)=\int_0^1 g(\dot\gamma(t),\dot\gamma(t))\,\diff t.
\]
This is a smooth $\R$-invariant function that satisfies the Palais-Smale condition. Its critical points may be of two kinds: ``genuine'' critical points with positive critical value, which are $I$-invariant geodesics of $(M,g)$ parametrized with constant positive speed, and ``spurious'' critical points with critical value zero, given by those curves that are constantly equal to some fixed point of the map $I$. Genuine critical points $\gamma$ come in critical orbits $\orb(\gamma)=\R\cdot\gamma$, which are homeomorphic to either a circle if $\gamma$ is a periodic curve, or the real line if $\gamma$ is an open curve. 

The critical orbits of open $I$-invariant geodesics are non-isolated ones, as it was proved by Grove in \cite[Theorem~2.4]{Grove:Isometry_invariant_geodesics}. This important fact implies that, when studying the multiplicity of isometry-invariant geodesics, one can always assume that all such geodesics are periodic ones. Any periodic $I$-invariant geodesic gives rise to an infinite sequence of critical orbits of the energy: indeed, the geodesic can be parametrized with some $\gamma\in\Lambda(M;I)$ with minimal period $p\geq1$, so that $\gamma(t)=\gamma(t+p)$; for every positive integer $m$, the curve $\gamma^{mp+1}\in\Lambda(M;I)$ given by $\gamma^{mp+1}(t)=\gamma((mp+1)t)$ is also a critical point of $E$, but it is equivalent to $\gamma$ from the geometric point of view.

\subsection{Topology of the space of invariant paths}
\label{s:Topology_of_Lambda_M_I}

The easiest way to find critical points of a bounded from below function satisfying the Palais-Smale condition is by minimizing it on any connected component of its domain. For the energy function $E$, such a procedure may end up giving a critical point with zero critical value, which is simply a fixed point of the isometry $I$. However, this happens only in finitely many connected components of the space of $I$-invariant paths.

\begin{lem}\label{l:stationary_curves}
The critical points of the energy function $E:\Lambda(M;I)\to\R$ with zero critical value are contained in finitely many connected components of $\Lambda(M;I)$.
\end{lem}

\begin{proof}
We only have to consider the case in which $I$ possesses some fixed points, otherwise the statement is trivially true. We denote by $B(q,r)$ the Riemannian ball in $(M,g)$ that is centered at $q$ and has radius $r$, and by $\injrad(M,g)$ the injectivity radius of our closed Riemannian manifold. For each $q\in\fix(I)$ there exists $r_q\in(0,\injrad(M,g))$ such that all points $p\in B(q,r_q)$ lie at distance less than $\injrad(M,g)$ from their image $I(p)$. For all such points $p$ we consider the unique geodesic $\gamma_p:\R\to M$ parametrized with constant speed such that $\gamma_p|_{[0,1]}$ joins $p$ with $I(p)$ and has length equal to their distance. Notice that, if $p\in\fix(I)$, its associated curve $\gamma_p$ is stationary at $p$. Since $\fix(I)$ is compact, we can find finitely many points $q_1,...,q_n\in\fix(I)$ such that the union of the balls $B(q_i,r_{q_i})$ covers $\fix(I)$. Consider the map $\Gamma:\bigcup_{i=1}^n B(q_i,r_{q_i})\to\Lambda(M;I)$ given by $\Gamma(p)=\gamma_p$. This map is clearly continuous, and thus sends each Riemannian ball $B(q_i,r_{q_i})$ into some connected component of $\Lambda(M;I)$. Since the stationary $I$-invariant curves (which are the fixed points of $I$) are contained in the image of $\Gamma$, they are contained in finitely many connected components of $\Lambda(M;I)$.
\end{proof}

As it was first remarked by Grove~\cite[Lemma~3.6]{Grove:Condition_C_for_the_energy_integral_on_certain_path_spaces_and_applications_to_the_theory_of_geodesics}, the homotopy type of the space of $I$-invariant curves depends only on the homotopy class of $I$ within the space of continuous map of $M$ into itself. In particular, if $I$ is homotopic to the identity, $\Lambda(M;I)$ is homotopy equivalent to the free loop space $\Lambda(M;\id)$. Given a continuous homotopy $I_t:M\to M$, with $I_0=\id$ and $I_1=I$, an explicit homotopy equivalence $\iota:\Lambda(M;\id)\to\Lambda(M;I)$ can be constructed as
\begin{align}
\label{e:homotopy_equivalence_of_path_spaces}
\iota(\gamma)(t)
=
\left\{
  \begin{array}{ll}
    \gamma(2t) & \mbox{ if }t\in\big[0,\tfrac12\big],\vspace{5pt} \\
    I_{2t-1}(\gamma(0)) & \mbox{ if }t\in\big[\tfrac12,1\big].
  \end{array}
\right.
\end{align}
A homotopy inverse is the map  $\nu:\Lambda(M,I)\to\Lambda(M;\id)$ given by
\begin{align}
\label{e:inverse_homotopy_equivalence_of_path_spaces}
\nu(\gamma)(t)
=
\left\{
  \begin{array}{ll}
    \gamma(2t) & \mbox{ if }t\in\big[0,\tfrac12\big],\vspace{5pt} \\
    I_{2-2t}(\gamma(0)) & \mbox{ if }t\in\big[\tfrac12,1\big].
  \end{array}
\right.
\end{align}

\begin{rem}\label{r:broken_geodesics}
Actually, in order for $\iota(\gamma)$ to lie in  $\Lambda(M;I)$, all the curves of the form $t\mapsto I_t(q)$ must have regularity $\Wloc$. One can achieve this by an arbitrarily small $C^0$-perturbation of the homotopy $I_t$, for instance as follows. Choose an integer $n$ that is large enough so that, for all $q\in M$ and $s,t\in[0,1]$ with $|s-t|\leq 1/n$, the Riemannian distance between $I_{s}(q)$ and $I_{t}(q)$ is less than the injectivity radius of $(M,g)$. Let $\gamma_q:[0,1]\to M$ be the (unique) continuous curve such that, for each $i\in\{0,...,n-1\}$, its restriction $\gamma_{q}|_{[i/n,(i+1)/n]}$ is the shortest geodesic parametrized with constant speed joining $I_{i/n}(q)$ and $I_{(i+1)/n}(q)$. We can now reset the homotopy to be $I_t(q):=\gamma_q(t)$. \hfill\qed
\end{rem}

From now on, let us assume that the fundamental group $\pi_1(M)$ is abelian. This implies that there is a bijective map $\pi_1(M)\to\pi_0(\Lambda(M;\id))$ that sends any homotopy class of loops to the corresponding connected component. By taking the composition with the map $\pi_0(\iota)$ induced by the homotopy equivalence~\eqref{e:homotopy_equivalence_of_path_spaces} we obtain a bijective map
\begin{align}\label{e:bijection_pi1_pi0}
\pi_1(M)\toup^{\cong}\pi_0(\Lambda(M;I)).
\end{align}
Consider the evaluation map that sends every loop $\gamma$ to its base point $\gamma(0)$. Since the fundamental group of $M$ is abelian, this map induces a surjective homomorphism of the fundamental group of any connected component of the free loop space $\Lambda(M;\id)$ onto the fundamental group of $M$. By precomposing with the isomorphism $\pi_1(\nu)$ induced by the homotopy equivalence~\eqref{e:inverse_homotopy_equivalence_of_path_spaces}, we also get a surjective homomorphism of the fundamental group of any connected component of $\Lambda(M;I)$ onto the one of $M$. More precisely, for all $q_0\in M$ and $\gamma_0\in\Lambda(M;I)$ with $\gamma_0(0)=q_0$, the evaluation map $\ev:\Lambda(M;I)\to M$ given by $\ev(\gamma)=\gamma(0)$ induces a surjective homomorphism
\begin{align}\label{e:epimorphism_pi1_pi1}
\pi_1(\ev):\pi_1(\Lambda(M;I),\gamma_0)\to\pi_1(M,q_0).
\end{align}

\subsection{Local properties of isometry-invariant geodesics}\label{s:local_properties}

Let $\gamma\in\Lambda(M;I)$ be an $I$-invariant geodesic that is periodic with minimal period $p\geq1$ (for the remaining of section~\ref{s:Preliminaries} we do not need to assume the isometry $I$ to be homotopic to the identity). The local Morse-theoretic properties of the sequence of critical orbits $\{\orb(\gamma^{mp+1})\ |\ m\in \N\}$ were thoroughly investigated by Grove and Tanaka \cite{Grove_Tanaka:On_the_number_of_invariant_closed_geodesics_ACTA, Tanaka:On_the_existence_of_infinitely_many_isometry_invariant_geodesics}. Here, we shall recall only the ones that will be needed for the proof of our multiplicity results in section~\ref{s:multiplicity_results} (for a quick reference, see~\cite[Sections~2.2 and~3.1]{Mazzucchelli:On_the_multiplicity_of_isometry_invariant_geodesics_on_product_manifolds}). 

We recall that the Morse index $\ind(\gamma)$ is the (finite) dimension of the negative eigenspace of the Hessian of $E$ at the critical point. The behavior of the index along our sequence of critical orbits is described by the following lemma.
\begin{lem}[Grove-Tanaka]
\label{l:iteration_formulae}
The sequence $m^{-1}\ind(\gamma^{mp+1})$ converges to a non-negative real number $\avind(\gamma)$ as $m\to\infty$. Moreover, if this limit is zero, then $\ind(\gamma^{mp+1})=0$ for all $m\in\N$. 
\hfill\qed
\end{lem}

In order to discuss further properties, it is more convenient (and formally equivalent) to see each critical orbit $\orb(\gamma^{mp+1})$ of our sequence as the same critical orbit $\orb(\gamma)$ living in different spaces. For all $\tau>0$ we introduce the space of invariant paths
\[
\Lambda^{\tau}(M;I)=\{\zeta\in\Wloc(\R;M)\ |\ I(\zeta(t))=\zeta(t+\tau)\ \forall t\in\R\},
\]
and the average energy $E^\tau:\Lambda^{\tau}(M;I)\to\R$ given by
\[
E^\tau(\zeta)=\frac1\tau \int_0^\tau g(\dot\zeta(t),\dot\zeta(t))\,\diff t.
\]
In order to go forth and back from this new setting to the former one, we can employ the diffeomorphism $\Psi^\tau:\Lambda^\tau(M;I)\to\Lambda(M;I)$ given by $\Psi^\tau(\zeta)=\zeta^\tau$, which relates the energies as $E\circ \Psi^\tau=\tau^2 E^\tau$. The curve $\gamma^{mp+1}\in\Lambda(M;I)$, which is a critical point of $E$, corresponds to the curve $\gamma\in\Lambda^{mp+1}(M;I)$, which is a critical point of $E^{mp+1}$. The next lemma tells us that, if we want to understand the behavior of the function $E^{mp+1}$ near the critical	point $\gamma^{mp+1}$ for all $m\in\N$, it is enough to do it for finitely many values of $m$.
\begin{lem}[Grove-Tanaka]
\label{l:Omegas}
There exists finitely many spaces $\Omega_1,...,\Omega_n$, integers $p_1,...,p_n$ that are multiples of the period $p$, and a partition $\N=\M_1\cup\M_2\cup...\cup\M_n$ with the following properties:
\begin{itemize}
\item[(i)] each space $\Omega_i$ contains $\gamma$, is contained in  
\[\Lambda^{p_i}(M;I)\cap\bigcap_{m\in\M_i}\Lambda^{mp+1}(M;I)\]
as a Hilbert submanifold of each of the spaces involved in the intersection, and is invariant by the gradient flow of $E^{mp+1}$ for all $m\in\M_i$;
\item[(ii)] we have $E^{p_i}|_{\Omega_i}=E^{mp+1}|_{\Omega_i}$ for all $m\in\M_i$;
\item[(iii)] the null space of the Hessian of $E^{mp+1}$ at $\gamma$ is the same as the null space of the Hessian of restricted function $E^{mp+1}|_{\Omega_i}$ at $\gamma$ for all $m\in\M_i$. 
\hfill\qed
\end{itemize}

\end{lem}

\subsection{Bangert's homotopy and isometry-invariant geodesics}\label{s:Bangert}
A sufficiently iterated closed geodesic cannot be detected by a 1-dimensional minimax scheme. This fundamental remark was first made by Bangert \cite{Bangert:Closed_geodesics_on_complete_surfaces}, and played a crucial role in the proof of many existence results in Riemannian geometry (see e.g.~\cite{Ballmann_Thorbergsson_Ziller:Closed_geodesics_and_the_fundamental_group, Bangert_Klingenberg:Homology_generated_by_iterated_closed_geodesics,Bangert_Hingston:Closed_geodesics_on_manifolds_with_infinite_Abelian_fundamental_group, Bangert:On_the_existence_of_closed_geodesics_on_two_spheres, Abbondandolo_Macarini_Paternain:On_the_existence_of_three_closed_magnetic_geodesics_for_subcritical_energies}). A detailed account that employs our notation can be found in~\cite[Section~3.2]{Mazzucchelli:On_the_multiplicity_of_isometry_invariant_geodesics_on_product_manifolds}.

\begin{lem}[Bangert]
\label{l:Bangert}
Let $\Theta:[0,1]\to\Lambda^p(M;\id)$ be a continuous path. For each $m\in\N$, let $\iota^m:\Lambda^p(M;\id)\hookrightarrow\Lambda^{mp}(M;\id)$ denote the inclusion. The path $\iota^m\circ\Theta$ can be homotoped, fixing its endpoints, to a new path $\Theta_m:[0,1]\to\Lambda^{mp}(M;\id)$ such that
\[
E^{mp}(\Theta_{m}(s))\leq \max\big\{ E^{p}(\Theta(0)), E^p(\Theta(1)) \big\} + \frac{\mathrm{const}}{m},
\qquad
\forall s\in[0,1],
\]
where $\mathrm{const}>0$ is a quantity independent of $m$. Moreover, there is a continuous function $\sigma_m:[0,1]\to[0,1]$ such that $\Theta_m(s)(0)=\Theta(\sigma_m(s))(0)$ for all $s\in[0,1]$.
\hfill\qed
\end{lem}

By means of Bangert's Lemma and the results of section~\ref{s:local_properties}, we obtain the following important application in the context of isometry-invariant geodesics.

\begin{lem}\label{l:Bangert_for_I_invariant_geodesics}
Let $\gamma\in\Lambda(M;I)$ be an $I$-invariant geodesic that is periodic with minimal period $p\geq1$. For all integers $m$ large enough, every  path $\Theta:[-1,1]\to\Lambda(M;I)$ such that $\Theta(0)=\gamma^{mp+1}$ and $E(\Theta(s))<E(\gamma^{mp+1})$ for all $s\neq0$ can be homotoped, with fixed endpoints, inside the sublevel set $\{E<E(\gamma^{mp+1})\}$.
\end{lem}

\begin{proof}
If $\avind(\gamma)>0$, the statement is easy to prove: according to Lemma~\ref{l:iteration_formulae}, the Morse index $\ind(\gamma^{mp+1})$ becomes large if $m$ is large, and thus we obtain the desired homotopy by pushing our path in a suitable negative direction of the Hessian of $E$ near $\gamma^{mp+1}$. Let us now focus on the other, hard, case: $\avind(\gamma)=0$. We work in the setting of section~\ref{s:local_properties}: we see $\gamma^{mp+1}\in\Lambda(M;I)$ as the curve $\gamma\in\Lambda^{mp+1}(M;I)$ with average energy
\[
c:=E^{mp+1}(\gamma)=\frac{1}{mp+1}\int_0^{mp+1} g(\dot\gamma(t),\dot\gamma(t))\,\diff t=g(\dot\gamma(0),\dot\gamma(0)).
\]
Consider the partition $\N=\M_1\cup\M_2\cup...\cup\M_n$ given by Lemma~\ref{l:Omegas}. We choose any $\M_i$ and the corresponding space $\Omega_i$. We recall that all $\zeta\in\Omega_i$ are $p_i$-periodic curves, where $p_i$ is some multiple of $p$. We fix $r>0$ smaller than half the injectivity radius of $(M,g)$, and we define the following $C^0$-neighborhood of $\gamma$ in $\Omega_i$
\[
\Upsilon:=
\big\{
\zeta\in\Omega_i\ \big|\ \dist(\zeta(t),\gamma(t))<r\ \ \forall t\in\R
\big\}.
\]
Consider the sublevel set $\Upsilon^-:=\{E^{p_i}|_{\Upsilon}<c\}\subset\Upsilon$. We denote by $\Upsilon_1^-,...,\Upsilon_r^-$ the path-connected components of $\Upsilon^-$. Notice that they are finitely many: $r-1$ is bounded from above by the rank of the local homology  
$\Hom_1(\Upsilon^-\cup\{\gamma\},\Upsilon^-;\Z)$,
which is finite. For all pair of distinct values $\alpha,\beta\in\{1,...,r\}$ we fix a continuous path 
\begin{align}
\label{e:standard_path}
\Psi_{\alpha\beta}:[-1,1]\to\Upsilon^-\cup\{\gamma\}
\end{align}
such that $\Psi_{\alpha\beta}(0)=\gamma$, $\Psi_{\alpha\beta}|_{[-1,0)}$ is contained in $\Upsilon_\alpha^-$, whereas $\Psi_{\alpha\beta}|_{(0,1]}$ is contained in $\Upsilon_\beta^-$.

Let $\nu^m:\Omega_i\hookrightarrow\Lambda^{mp+1}(M;I)$ be the inclusion, where $m\in\M_i$. We claim that, for all $\Psi\in\{\Psi_{\alpha\beta}\ |\ \alpha\neq\beta\}$ and $m$ sufficiently large, $\nu^m\circ\Psi$ can be homotoped, with fixed endpoints, to a path that is contained in the sublevel set $\{E^{mp+1}<c\}$. This can be seen as follows. We denote by $m_0\in\N$ the largest integer such that $m_0p_i<mp+1$, and by $\iota^{m_0}:\Omega_i\to\Lambda^{m_0 p_i}(M;\id)$ the inclusion. By Lemma~\ref{l:Bangert}, the path $\iota^{m_0}\circ\Psi$ can be homotoped, with fixed endpoints, to a path $\Psi_{m_0}$ such that
\begin{align}\label{e:Bangert_estimate_in_use}
E^{m_0 p_i}(\Psi_{m_0}(s))\leq \max\big\{ E^{p_i}(\Psi(0)), E^{p_i}(\Psi(1)) \big\} + \frac{\mathrm{const}}{m_0},
\qquad
\forall s\in[0,1].
\end{align}
Moreover, we have $\Psi_{m_0}(s)(0)=\Psi(\sigma_{m_0}(s))(0)$ for a suitable continuous function $\sigma_{m_0}:[0,1]\to[0,1]$. We define the continuous path $\Psi_{m}':[-1,1]\to\Lambda^{mp+1}(M;I)$ to be
\[
\Psi_m'(s)(t):=
\left\{
  \begin{array}{lll}
    \Psi_{m_0}(s)(t) &  & \mbox{if }t\in[0,m_0p_i], \vspace{3pt} \\ 
    \Psi(\sigma_{m_0}(s))(t) &  & \mbox{if }t\in[m_0p_i,mp+1]. \\ 
  \end{array}
\right.
\]
This path is clearly homotopic, with fixed endpoints, to $\nu^m\circ\Psi$: indeed, it suffices to extend in the obvious way the homotopy from $\iota^{m_0}\circ\Psi$ to $\Psi_{m_0}$. The average energy of the curves along this path can be estimated as
\begin{align*}
E^{mp+1}(\Psi_m'(s))
=\,
& \frac{1}{mp+1}
\Bigg[
\int_0^{m_0p_i} g\big(\tfrac{\diff}{\diff t}\Psi_{m_0}(s)(t),\tfrac{\diff}{\diff t}\Psi_{m_0}(s)(t)\big)\,\diff t \\
& +
\int_{m_0p_i}^{mp+1} g\big(\tfrac{\diff}{\diff t}\Psi(\sigma_{m_0}(s))(t),\tfrac{\diff}{\diff t}\Psi(\sigma_{m_0}(s))(t)\big)\,\diff t
\Bigg]\\
\leq\,
& \frac{1}{mp+1}
\Bigg[
\int_0^{m_0p_i} g\big(\tfrac{\diff}{\diff t}\Psi_{m_0}(s)(t),\tfrac{\diff}{\diff t}\Psi_{m_0}(s)(t)\big)\,\diff t \\
& +
\int_{m_0p_i}^{(m_0+1)p_i} g\big(\tfrac{\diff}{\diff t}\Psi(\sigma_{m_0}(s))(t),\tfrac{\diff}{\diff t}\Psi(\sigma_{m_0}(s))(t)\big)\,\diff t
\Bigg]\\
=\,
& \frac{1}{mp+1} \Big( m_0p_i E^{m_0p_i}(\Psi_{m_0}(s))   +  p_i E^{p_i}(\Psi(\sigma_{m_0}(s)))\Big),
\end{align*}
where for the  inequality we have used the fact that $mp+1-m_0p_i<p_i$. By~\eqref{e:Bangert_estimate_in_use} and since $E^{p_i}(\Psi(\sigma_{m_0}(s)))<c$ we further obtain
\begin{align*}
E^{mp+1}(\Psi_m'(s))
\leq\,
& \frac{m_0p_i}{mp+1} \Big( \max\big\{ E^{p_i}(\Psi(0)), E^{p_i}(\Psi(1)) \big\} + \frac{\mathrm{const}}{m_0}  \Big)
+ \frac{p_i\, c}{mp+1}\\
\leq\,
& \max\big\{ E^{p_i}(\Psi(0)), E^{p_i}(\Psi(1)) \big\} + \frac{p_i\, \mathrm{const}}{mp+1}  
+ \frac{p_i\, c}{mp+1}.
\end{align*}
The second and third summand in the previous line go to zero as $m\to\infty$, whereas both $E^{p_i}(\Psi(0))$ and $E^{p_i}(\Psi(1))$ are bounded from above by $c-\delta$, for some $\delta>0$. Therefore $E^{mp+1}(\Psi_m'(s))<c$ provided $m\in\M_i$ is large enough.

In order to conclude the proof of the lemma, all is left to do is to establish the following claim: let $\Theta:[-1,1]\to\Lambda^{mp+1}(M;I)$ be any continuous path with $\Theta(0)=\gamma$ and $E^{mp+1}(\Theta(s))<c$ for all $s\neq0$; then, there is a homotopy, with fixed endpoints, from $\Theta$ to a new path $\Theta'$ such that $E^{mp+1}(\Theta'(s))<c$ for all $s\neq0$, and the restriction $\Theta'|_{[-\epsilon,\epsilon]}$, for some $\epsilon>0$, coincides with  one of our standard paths~\eqref{e:standard_path}. 

Let us prove this claim.  Consider a characteristic manifold $\Upsilon'\subset\Upsilon$ for the function $E^{mp+1}|_{\Upsilon}$ at $\gamma$. Namely, $\Upsilon'$ is a finite-dimensional submanifold diffeomorphic to a ball, containing the critical point $\gamma$ in its interior, whose tangent space $\Tan_\gamma\Upsilon'$ is the nullspace of the Hessian of $E^{mp+1}$ at $\gamma$, and such that the gradient of $E^{mp+1}$ is tangent to $\Tan\Upsilon'$ along $\Upsilon'$. Since $\avind(\gamma)=0$, Lemma~\ref{l:iteration_formulae} tells us that the Morse index of $E^{mp+1}$ at $\gamma$ is zero. According to the generalized Morse Lemma (see e.g.~\cite[page~501]{Gromoll_Meyer:Periodic_geodesics_on_compact_Riemannian_manifolds} or \cite[page~72]{Chang:Infinite_dimensional_Morse_theory_and_multiple_solution_problems}), $\Upsilon'$ has a tubular neighborhood $\Upsilon''\subset\Lambda^{mp+1}(M;I)$ and a diffeomorphism $\Phi:\Upsilon'\times B\to\Upsilon''$ (where $B$ is a Hilbert ball) that maps the zero section $\Upsilon'\times\{0\}$ to $\Upsilon'$ and such that
\[
E^{mp+1}\circ\Phi(\zeta',x)=E^{mp+1}(\zeta') + \|x\|^2,\qquad\forall(\zeta',x)\in\Upsilon'\times B.
\] 
In particular, we have a continuous deformation retraction $r_t:\Upsilon''\to\Upsilon''$ given by 
\[r_t\circ\Phi(\zeta',x)=(\zeta',(1-t)x).\]
Notice that the deformation decreases the value of $E^{mp+1}$ (away from the characteristic manifold $\Upsilon'$). Let $\epsilon>0$ be such that the portion of the path $\Theta|_{[-3\epsilon,3\epsilon]}$ is contained inside $\Upsilon''$. We take a continuous function $\rho:[-3\epsilon,3\epsilon]\to[0,1]$ supported inside $(-3\epsilon,3\epsilon)$ and such that $\rho|_{[-2\epsilon,2\epsilon]}\equiv1$, and we replace  $\Theta|_{[-3\epsilon,3\epsilon]}$ by the homotopic path $s\mapsto r_{\rho(s)}\circ\Theta(s)$. This shows that we can assume our given path $\Theta$ to satisfy $\Theta(s)\in\Upsilon'$ for all $s\in[-2\epsilon,2\epsilon]$. In particular $\Theta(-2\epsilon)\in\Upsilon^-_\alpha$ and $\Theta(2\epsilon)\in\Upsilon^-_\beta$ for some $\alpha,\beta$. Let us choose a continuous path $\Theta_\alpha:[-2\epsilon,-\epsilon]\to\Upsilon_\alpha^-$ joining $\Theta(-2\epsilon)$ and $\Psi_{\alpha\beta}(-1)$, and another continuous path  $\Theta_\beta:[-2\epsilon,-\epsilon]\to\Upsilon_\beta^-$ joining $\Psi_{\alpha\beta}(1)$ and $\Theta(2\epsilon)$. We  concatenate $\Theta_\alpha$, $\Psi_{\alpha\beta}$, and $\Theta_\beta$ in order to obtain the continuous path $\Theta':[-2\epsilon,2\epsilon]\to\Upsilon$, i.e.
\[
\Theta'(s)=
\left\{
  \begin{array}{lll}
    \Theta_\alpha(s) &  & \mbox{if } s\in[-2\epsilon,-\epsilon], \vspace{3pt} \\ 
    \Psi_{\alpha\beta}(s/\epsilon) &  & \mbox{if } s\in[-\epsilon,\epsilon], \vspace{3pt}  \\ 
    \Theta_\beta(s) &  & \mbox{if } s\in[\epsilon,2\epsilon].
  \end{array}
\right.
\]
Since both paths $\Theta|_{[-2\epsilon,2\epsilon]}$ and $\Theta'|_{[-2\epsilon,2\epsilon]}$ lie inside the small open set $\Upsilon$,  there exists a homotopy $H_t:[-2\epsilon,2\epsilon]\to\Lambda^{mp+1}(M;I)$ such that $H_0=\Theta|_{[-2\epsilon,2\epsilon]}$, $H_1=\Theta'|_{[-2\epsilon,2\epsilon]}$, and all the paths $H_t$ have the same endpoints. One such map can be constructed as follows. Assume for the sake of simplicity that $\Theta(s)$ and $\Theta'(s)$ are piecewise-smooth $I$-invariant curves for  all $s\in(-2\epsilon,2\epsilon)$. We can always achieve this by  $C^0$-small homotopies of the paths $\nu^m\circ\Theta|_{[-2\epsilon,2\epsilon]}$ and $\nu^m\circ\Theta'|_{[-2\epsilon,2\epsilon]}$ with fixed endpoints, where $\nu^m:\Omega_i\hookrightarrow\Lambda^{mp+1}(M;I)$ denotes the inclusion as before. We set $t\mapsto H_t(s)(t')$ to be the shortest geodesic parametrized with constant speed joining the points $\Theta(s)(t')$ and $\Theta'(s)(t')$. Since
\[\dist(\Theta(s)(t'),\Theta'(s)(t')) < 2r <\injrad(M,g),\qquad\forall s\in[-2\epsilon,2\epsilon],\]
such a geodesic is unique. Moreover, since the isometry $I$ maps geodesics to geodesics, we have $I(H_t(s)(0))=H_t(s)(mp+1)$. Therefore, $H_t$ is well defined as a homotopy of paths in $\Lambda^{mp+1}(M;I)$. Finally, we extend $\Theta'$ to a continuous path of the form $\Theta':[0,1]\to\Lambda^{mp+1}(M;I)$ by setting $\Theta'\equiv\Theta$ outside $[-2\epsilon,2\epsilon]$. This proves our claim, and concludes the proof of the lemma.
\end{proof}

\section{Multiplicity results}\label{s:multiplicity_results}

\subsection{Fundamental groups of rank at least two}\label{s:multiplicity_results_1}

Throughout this section we will consider a closed connected Riemannian manifold $(M,g)$ whose fundamental group is abelian, equipped with an isometry $I$ that is homotopic to the identity. If the rank of the fundamental group is larger than one, we can easily find infinitely many closed geodesics on the manifold: it suffices to pick two linearly independent elements of infinite order $z_1,z_2\in\pi_1(M)$, and to consider the closed geodesic $\gamma_k$ of minimal length among all the curves freely homotopic to a representative of\footnote{Here and in the following, we employ the additive notation for the abelian fundamental group.} $k z_1+z_2$; all curves $\gamma_k$ and $\gamma_j$, for $k\neq j$, are geometrically distinct. The situation is analogous for the case of $I$-invariant geodesics, even though we must be slightly more clever in the selection of the elements of the fundamental group employed to carry over the minimizations. 

\begin{thm}\label{t:rank_larger_than_one}
Let $(M,g)$ be a closed connected Riemannian manifold whose fundamental group is abelian of rank at least two. Every isometry $I$ homotopic to the identity has infinitely many $I$-invariant geodesics.
\end{thm}

\begin{proof}
The fundamental group of our Riemannian manifold has the form $\pi_1(M)\cong\Z^{2}\oplus H$, for some abelian group $H$. Fix two generators $z_1,z_2$ of $\Z^{2}\subset\pi_1(M)$. For all $\kk=(k_1,k_2)\in\Z^2$, we denote by $\Omega_{\kk}$ the connected component of the path space $\Lambda(M;I)$ corresponding to $k_1z_1+k_2z_2$ under the bijection~\eqref{e:bijection_pi1_pi0}.  Notice in particular that distinct $\kk$'s give rise to distinct $\Omega_{\kk}$'s.  We denote by $\orb(\gamma_{\kk})$ a global minimum of the energy function $E$ on the connected component $\Omega_{\kk}$. By Lemma~\ref{l:stationary_curves}, for all but finitely many $\kk$'s the curve $\gamma_{\kk}$ is not stationary, and thus it is a genuine $I$-invariant geodesic. We claim that the family  $\{\gamma_{\kk}\,|\,\kk\in\Z^2\}$ contains infinitely many (geometrically distinct) $I$-invariant geodesics. Assume by contradiction that our claim does not hold. In particular, there exists an infinite subset $\K_1\subset\N$, a curve $\gamma_1\in\Lambda(M;I)$ that is periodic with minimal period $p_1\geq1$, and a  function $\mu_1:\K_1\to\N$ such that
\begin{align*}
\gamma_{\kk}=\gamma_1^{\mu_1(k_2)p_1+1},\qquad\forall \kk=(k_1,k_2)\in\{1\}\times\K_1.
\end{align*}
We proceed iteratively for increasing values of $n\geq2$: there exists an infinite subset $\K_n\subset\K_{n-1}$, a curve $\gamma_n\in\Lambda(M;I)$ that is periodic with minimal period $p_n\geq1$, and a  function $\mu_n:\K_n\to\N$ such that
\begin{align}\label{e:iterated_curve}
\gamma_{\kk}=\gamma_n^{\mu_n(k_2)p_n+1},\qquad\forall \kk=(k_1,k_2)\in\{n\}\times\K_n.
\end{align}
By the pigeonhole principle, there exist integers $n_1<n_2$ such that $\gamma_{n_1}=\gamma_{n_2}=:\gamma$ (and thus $p_{n_1}=p_{n_2}=:p$). We take two integers $m_1,m_2\in\K_{n_2}$ with $m_1<m_2$. Now, consider the homotopy equivalence $\nu:\Lambda(M;I)\to\Lambda(M;\id)$ defined in~\eqref{e:inverse_homotopy_equivalence_of_path_spaces}. Notice that the effect of applying $\nu$ to an $I$-invariant curve $\zeta$ is to concatenate $\zeta|_{[0,1]}$ with a path that depends only on $\zeta(0)$ and the homotopy $I_t$. By equation~\eqref{e:iterated_curve}, we have
\begin{equation}\label{e:z1_factor}
\begin{split}
(n_1-n_2)z_1
&=
\big[\nu\big(\gamma^{\mu_{n_1}(m_1)p+1}\big)\big]-\big[\nu\big(\gamma^{\mu_{n_2}(m_1)p+1}\big)\big]\\
&=
(\mu_{n_1}(m_1)-\mu_{n_2}(m_1))[\gamma|_{[0,p]}], 
\end{split}
\end{equation}
and analogously
\begin{equation}\label{e:z2_factor}
\begin{split}
(m_1-m_2)z_2
&=
\big[\nu\big(\gamma^{\mu_{n_1}(m_1)p+1}\big)\big]-\big[\nu\big(\gamma^{\mu_{n_1}(m_2)p+1}\big)\big]\\
&=
(\mu_{n_1}(m_1)-\mu_{n_1}(m_2))[\gamma|_{[0,p]}].
\end{split}
\end{equation}
The loop $\gamma|_{[0,p]}$ represents an element of the fundamental group of the form 
\[
[\gamma|_{[0,p]}]=w_1 z_1+w_2 z_2+h,
\] for some $w_1,w_2\in\Z$ and $h\in H$. Equation~\eqref{e:z1_factor} implies that $w_1\neq0$ and $w_2=0$. This contradicts equation~\eqref{e:z2_factor}, which implies $w_1=0$ and $w_2\neq0$.
\end{proof}

\subsection{Fundamental groups of rank at least one}\label{s:multiplicity_results_2}

Theorem~\ref{t:rank_larger_than_one} holds under the weaker assumption that the fundamental group is infinite abelian but not cyclic. This is the statement of our main result, Theorem~\ref{t:main}, which we restate here for the reader's convenience.
\vspace{5pt}

\noindent\textbf{Theorem~\ref{t:main}}. \textsl{Let $(M,g)$ be a closed connected Riemannian manifold whose fundamental group is isomorphic to $\Z\oplus H$, for some non-trivial abelian group $H$. Every isometry of $(M,g)$ homotopic to the identity admits infinitely many invariant geodesics.}
\vspace{5pt}

The situation here is much harder than in the previous subsection: since the rank of the fundamental group may be one, in general we cannot expect to find infinitely many $I$-invariant geodesics that are minima of the energy function in some connected component of the space of $I$-invariant curves. Instead, as in the special case of closed geodesics \cite{Bangert_Hingston:Closed_geodesics_on_manifolds_with_infinite_Abelian_fundamental_group}, we need to consider minimax critical orbits of the energy. 

Let $z$ be a generator of the subgroup $\Z\subset\pi_1(M)$. For all $k\in\N$, we denote by $\Omega_{k}$ the connected component of the path space $\Lambda(M;I)$ corresponding to $k z$ under the bijection~\eqref{e:bijection_pi1_pi0}. In particular $\Omega_k$ and $\Omega_j$ are different connected components if $k\neq j$. Moreover, Lemma~\ref{l:stationary_curves} guarantees that, for all but finitely many $k$, $\Omega_k$ does not contain stationary curves, and therefore the critical points of $E|_{\Omega_k}$ are genuine $I$-invariant geodesics. We denote by $\orb(\alpha_k)$ a global minimum of the energy $E$ on the connected component $\Omega_k$. Since we are looking for infinitely many $I$-invariant geodesics, we  assume that all the critical orbits of the energy are isolated: in particular $\orb(\alpha_k)$ is a strict local minimum of $E$. 

If the family $\{\orb(\alpha_k)\,|\,k\in\N\}$ corresponds to infinitely many geometrically distinct $I$-invariant geodesics, we are done. Therefore, let us focus on the other case: there exists an infinite subset $\K\subseteq \N$, a function $\eta:\K\to\N$, and a periodic $I$-invariant geodesic $\alpha$ with minimal period $q\geq1$ such that
\[
\alpha_k=\alpha^{\eta(k)q+1},\qquad\forall k\in\K.
\]
Since the $\alpha_k$'s belong to pairwise distinct connected components of $\Lambda(M;I)$, the function $\eta$ must be injective. By applying the homotopy equivalence $\nu:\Lambda(M,I)\to\Lambda(M;\id)$ defined in~\eqref{e:inverse_homotopy_equivalence_of_path_spaces}, we obtain
\begin{align}\label{e:iteration_properties_of_alpha}
(k_1-k_2)z=[\nu(\alpha_{k_1})]-[\nu(\alpha_{k_2})]=(\eta(k_1)-\eta(k_2))[\alpha|_{[0,q]}],\qquad\forall k_1,k_2\in\K.
\end{align}
This implies that the homotopy class $[\alpha|_{[0,q]}]$ does not belong to the subgroup $H$ of the fundamental group $\pi_1(M)$. Up to replacing $\alpha$ by $\tau\cdot\alpha$ for some $\tau\in(0,q)$, we can assume that $\alpha(0)$ is not a point of self-intersection of $\alpha$, that is, we have $\alpha(0)=\alpha(t)$ if and only if $t$ is a multiple of the period $q$.

Fix a non-zero  $h\in H$ and, for all $k\in\K$, a corresponding $h_k\in\pi_1(\Lambda(M;I),\alpha_k)$ such that $\pi_1(\ev) h_k=h$, where $\pi_1(\ev)$ is the surjective homomorphism in~\eqref{e:epimorphism_pi1_pi1}. We consider the minimax value
\begin{align}\label{e:minimax}
c_k:=\min_{\big.\Theta\in h_k}\,\max_{\big.s\in\R/\Z} E(\Theta(s)).
\end{align}
Standard critical point theory (see e.g.~\cite{Klingenberg:Lectures_on_closed_geodesics, Chang:Infinite_dimensional_Morse_theory_and_multiple_solution_problems, Mazzucchelli:Critical_point_theory_for_Lagrangian_systems}) implies that $c_k$ is a critical value of the energy $E$. We claim that
\begin{align*}
c_k > E(\alpha_k),\qquad \forall k\in\K.
\end{align*}
Indeed, assume by contradiction that $c_k=E(\alpha_k)$. Since $\orb(\alpha_k)$ is a strict minimum of the energy $E$, there must be a $\Theta$ achieving the minimax~\eqref{e:minimax} and taking values inside the critical orbit $\orb(\alpha_k)$. Since $\alpha(0)$ is not a point of self-intersection of $\alpha$, the homotopy class $h=\pi_1(\ev)[\Theta]$ must be a multiple of $[\alpha|_{[0,q]}]$. This is a contradiction, since $[\alpha|_{[0,q]}]$ is not contained in the subgroup $H$.

For our purpose, it will be useful to single out a particularly simple $\Theta$ achieving the above minimax.

\begin{lem}\label{l:nice_loop}
There exists $\Theta:\R/\Z\to\Lambda(M;I)$ that achieves the minimax in~\eqref{e:minimax} and such that $\Theta(s)$ lies in the sublevel set $\{E<c_k\}$ except at finitely many $s\in(0,1)$ in which it goes through some  point in $E^{-1}(c_k)\cap\crit(E)$.
\end{lem}

\begin{proof}
We start by considering an arbitrary $\Theta_0$ achieving the minimax. By pushing down $\Theta_0$ with the anti-gradient flow of $E$, we can assume that $\Theta_0$ touches the level set $E^{-1}(c_k)$ only along critical orbits of $E$. Moreover, by homotopying $\Theta_0$ fixing the base point $\Theta_0(0)=\alpha_k$, we can always assume that all the curves $\Theta_0(s)$ are piecewise smooth (indeed, even broken geodesics). Let $\epsilon>0$ be a small enough real number. More specifically, we want $\epsilon$ to be smaller than half the injectivity radius of $(M,g)$, and such that for any pair of critical points $\gamma_1,\gamma_2\in\crit(E)\cap E^{-1}(c_k)$ belonging to distinct critical orbits there exists  $t\in\R$ (actually $t\in[0,1]$) such that $\dist(\gamma_1(t),\gamma_2(t))>\epsilon$. Fix $n\in\N$ large enough so that 
\[
\dist(\Theta_0(s')(t),\Theta_0(s'')(t))<\epsilon,
\qquad
\forall t\in\R,\ s',s''\in[0,1]\mbox{ with }|s'-s''|\leq \tfrac1n.
\]
For all $i\in\{0,...,n-1\}$ such that the path $\Theta_0|_{[i/n,(i+1)/n]}$ intersects the minimax level set, we denote by $a_i$ [resp.~$b_i$]  the minimal [resp.~the maximal]  $s\in\big[\tfrac in,\tfrac{i-1}n\big]$ such that $\Theta_0(s)\in\crit(E)\cap E^{-1}(c_k)$.  Notice that, by our choice of $n$ and $\epsilon$, the critical points $\Theta_0(a_i)$ and $\Theta_0(b_i)$ of the energy $E$ belong to the same critical orbit. For the other values of $i$ we may simply set $a_i=b_i:=\tfrac in$. Let $\tau_i$ be the real number with smallest absolute value such that $\Theta_0(b_i)=\tau_i\cdot\Theta_0(a_i)$, i.e.
\[\Theta_0(b_i)(t)=\Theta_0(a_i)(t+\tau_i),\qquad\forall t\in\R.\]
We define the continuous loop $\Theta_1:\R/\Z\to\Lambda(M;I)$ by
\begin{align*}
\Theta_1(s):=
\left\{
  \begin{array}{lll}
    \tfrac{(s-a_i)\tau_i}{b_i-a_i}\cdot\Theta_0(a_i) &  & \mbox{if }i\in\{0,...,n-1\},\ s\in(a_i,b_i), \vspace{5pt} \\ 
    \Theta_0(s) &  & \mbox{otherwise}. 
  \end{array}
\right.
\end{align*}
It is not hard to see that the paths $\Theta_0|_{[a_i,b_i]}$ and $\Theta_1|_{[a_i,b_i]}$ are homotopic with fixed endpoints (and thus $[\Theta_0]=[\Theta_1]=h_k$). A homotopy $H_t:[a_i,b_i]\to\Lambda(M;I)$ joining them is defined by setting $t\mapsto H_t(s)(t')$ to be the shortest geodesic parametrized with constant speed joining $\Theta_0(s)(t')$ and $\Theta_1(s)(t')$. Since
\[\dist(\Theta_0(s)(t'),\Theta_1(s)(t')) < 2\epsilon <\injrad(M,g),\]
such a geodesic is unique. Moreover, the isometry $I$ maps geodesics to geodesics, and therefore $I(H_t(s)(0))=H_t(s)(1)$. This shows that $H_t$ is a well defined homotopy of paths in $\Lambda(M;I)$.

The union $\bigcup_{i=1}^n \overline{(a_i,b_i)}$ can be written as the disjoint union of closed intervals $[c_i,d_i]\subset(0,1)$, for $i\leq r\leq n$. Let $\tau:[0,1]\to\R$ be a continuous function supported inside $(0,1)$ and such that $\Theta_1(s)=\tau(s)\cdot\Theta_1(c_i)$ for all $i\in\{1,...,r\}$, $s\in[c_i,d_i]$. We define the continuous loop $\Theta_2:\R/\Z\to\Lambda(M;I)$ by $\Theta_2(s):=(-\tau(s))\cdot\Theta_1(s)$. Clearly $[\Theta_1]=[\Theta_2]=h_k$. The loop $\Theta_2$ lies inside the sublevel set $\{E<c_k\}$ except along the intervals $[c_i,d_i]$, where it is constantly equal to some $\gamma_i\in\crit(E)\cap E^{-1}(c_k)$. By collapsing each of those intervals to a point, we obtain a loop $\Theta$ with the desired properties.
\end{proof}

For all $k\in\K$, we denote by $\Theta_k$ a loop satisfying the assumptions of Lemma~\ref{l:nice_loop}, and such that the finite set
\[
S_k:=\{ s\in\R/\Z\ |\ E(\Theta_k(s))=c_k \}
\]
has minimal cardinality. Theorem~\ref{t:main} is an immediate consequence of the following lemma.

\begin{lem}
The critical orbits of the points in $\bigcup_{k\in\K}\Theta_k(S_k)$ correspond to infinitely many geometrically distinct $I$-invariant geodesics
\end{lem}

\begin{proof}
We will prove the statement by contradiction: we assume that, up to replacing the set $\K$ by an infinite subset of it, there exists a periodic $I$-invariant geodesic $\gamma\in\Lambda(M;I)$ of minimal period $p\geq 1$, a function $\mu:\K\to\N$, and points $s_k\in S_k$ and $t_k\in\R$ such that
$t_k\cdot\Theta_k(s_k)=\gamma^{\mu(k)p+1}$ for all $k\in\K$. In order to simplify the notation, we can  assume that all the $t_k$'s are zero. Indeed, if this is not true, we consider a continuous function $\tau_k:\R/\Z\to\R$ such that $\tau_k(0)=0$ and $\tau_k(s_k)=t_k$, and we replace $\Theta_k$ by the path $s\mapsto\tau_k(s)\cdot\Theta_k(s)$. Therefore
\[
\Theta_k(s_k)=\gamma^{\mu(k)p+1},\qquad\forall k\in\K.
\]
By the same reasoning that we did in equation~\eqref{e:iteration_properties_of_alpha} for $\alpha$, we infer that the function $\mu$ is injective, and in particular $\mu(k)\to\infty$ as $k\to\infty$. Fix a small $\epsilon_k>0$, so that $s_k$ is the only point in the intersection $[s_k-\epsilon_k,s_k+\epsilon_k]\cap S_k$. By Lemma~\ref{l:Bangert_for_I_invariant_geodesics}, if $k$ is sufficiently large, the restricted path $\Theta_k|_{[s_k-\epsilon_k,s_k+\epsilon_k]}$ can be pushed, fixing its endpoints, inside the sublevel set $\{E<c_k\}$. This contradicts the minimality of the set $S_k$.
\end{proof}

\begin{rem}
\label{r:infinite_cyclic_fundamental_group}
As we wrote in the introduction, it is plausible that Theorem~\ref{t:main} still hold if the fundamental group of $M$ is isomorphic to $\Z$. Indeed, in that case, by generalizing an argument due to Bangert and Hingston \cite[Lemma~2]{Bangert_Hingston:Closed_geodesics_on_manifolds_with_infinite_Abelian_fundamental_group}, one would obtain non-trivial homotopy classes $h_k\in\pi_n(\Lambda(M;I),\alpha_k)$, for some $n\geq2$, such that $\pi_n(\ev)h_k=0$. These classes can be employed to setup an $n$-dimensional minimax scheme, thus obtaining a sequence of critical orbits $\orb(\gamma_k)$, where $\gamma_k$ is freely homotopic to $\alpha_k$. In order for our arguments to go through in this new situation, the only missing ingredient is an $n$-dimensional version of Lemma~\ref{l:nice_loop}, that is, for maps $\Theta:S^n\to\Lambda(M;I)$. Unfortunately, manipulating representatives of higher-dimensional homotopy classes seems to be a hard task, and significant new ideas may be needed.
\hfill\qed
\end{rem}

\bibliography{_biblio}
\bibliographystyle{amsalpha}

\end{document}